\documentclass[12pt]{amsart}
\usepackage{amsmath,latexsym,amssymb}
\usepackage{enumerate}
\usepackage{amsthm}
\usepackage{epsfig,graphicx}
\usepackage{tikz}
\usepackage{mathrsfs}

\newtheorem{theorem}{Theorem}[section]
\newtheorem{definition}[theorem]{Definition}
\newtheorem{corollary}[theorem]{Corollary}
\newtheorem{example}[theorem]{Example}
\newtheorem{lemma}[theorem]{Lemma}
\newtheorem{remark}[theorem]{Remark}
\newtheorem*{remarks*}{Remarks}
\newtheorem{proposition}[theorem]{Proposition}

\newcommand\NN{\mathbb{N}}
\newcommand\RR{\mathbb{R}}

\newcommand\EE{\mathbb{E}}
\newcommand\PP{\mathbb{P}}
\newcommand\eps{\varepsilon}
\newcommand\diam{\mathrm{diam}}

\newcommand\Graph{\text{Graph}}

\title[Box-like statistically self-affine functions]{Box-counting dimension and differentiability of box-like statistically self-affine functions}

\author[Pieter Allaart]{Pieter Allaart$^*$}
\address[P. Allaart]{Mathematics Department, University of North Texas, 1155 Union Cir \#311430, Denton, TX 76203-5017, U.S.A.}
\email{allaart@unt.edu}
\thanks{$^*$ Corresponding author.}

\author{Taylor Jones}
\address[T. Jones]{Mathematics Department, University of North Texas, 1155 Union Cir \#311430, Denton, TX 76203-5017, U.S.A.}
\email{RandallJones2@my.unt.edu}

\begin{document}

\begin{abstract}
We consider a class of ``box-like" statistically self-affine functions, and compute the almost-sure box-counting dimension of their graphs. Furthermore, we consider the differentiability of our functions, and prove that, depending on an explicitly computable functional of the model, they are almost surely either differentiable almost everywhere or non-differentiable almost everywhere.
\end{abstract}


\maketitle

\section{Introduction}

Let $S_1,\dots,S_n$ be an iterated function system (IFS) of affine contractions on $\RR^2$, whose images of $[0,1]^2$ are rectangles with sides parallel to the coordinate axes.
It is well know that the IFS admits an attractor $E$; that is, a unique nonempty compact set satisfying 
$$E=\bigcup_{i=1}^n S_i(E)$$
see \cite[Chapter 9]{falconer}. Following the literature we call the attractor of such an IFS a {\em box-like self-affine set}.  
Early examples of box-like self-affine sets are the generalized Sierpinski carpets of Bedford \cite{bedford1} and McMullen \cite{mcmullen}, whose Hausdorff and box dimensions these authors explicitly computed. Their work was generalized by Lalley and Gatzouras \cite{lalley}. In each of these papers, however, it is assumed that (i) each of the maps in the IFS contracts more sharply in the horizontal than in the vertical direction; and (ii) the images of the unit square under the maps ``line up" neatly in horizontal strips.
More recently, Fraser \cite{fraser} has dropped these assumptions and computed the box-counting dimension of box-like self-affine sets.

In this article, we will be looking at cases when box-like self-affine sets are the graphs of continuous functions, and adding an element of randomization.
In the deterministic setting, Bedford \cite{bedford2} computes the box-counting dimension of the graphs of such functions.
One particular class of deterministic box-like self-affine sets which are the graphs of continuous functions is known as Okamoto functions \cite{okamoto}.
For a fixed $\alpha\in(0,1)$, the graph of an Okamoto function $F_\alpha$ with parameter $\alpha$ is the attractor of the IFS
\begin{gather*}
S_0(x,y)=\begin{bmatrix} 1/3 & 0\\ 0 & \alpha \end{bmatrix} \begin{bmatrix}x\\y\end{bmatrix}, \qquad
S_1(x,y)=\begin{bmatrix} 1/3 & 0\\ 0 & 1-2\alpha \end{bmatrix}
\begin{bmatrix} x\\y \end{bmatrix}+\begin{bmatrix} 1/3\\ \alpha \end{bmatrix},\\
S_2(x,y)=\begin{bmatrix} 1/3 & 0\\ 0 & \alpha \end{bmatrix} 
\begin{bmatrix} x\\y \end{bmatrix}+\begin{bmatrix} 2/3\\ 1-\alpha \end{bmatrix}.
\end{gather*}
The family $\{F_\alpha\}$ has several well-known members, such as Perkins' function \cite{perkins} ($\alpha=5/6$), the Cantor ``devil's staircase" ($\alpha=1/2$) and even the identity function ($\alpha=1/3$). 
Pictured in Figure \ref{fig:Okamoto-graphs} are two examples of Okamoto's function $F_\alpha$.

\begin{figure}
\begin{center}
\epsfig{file=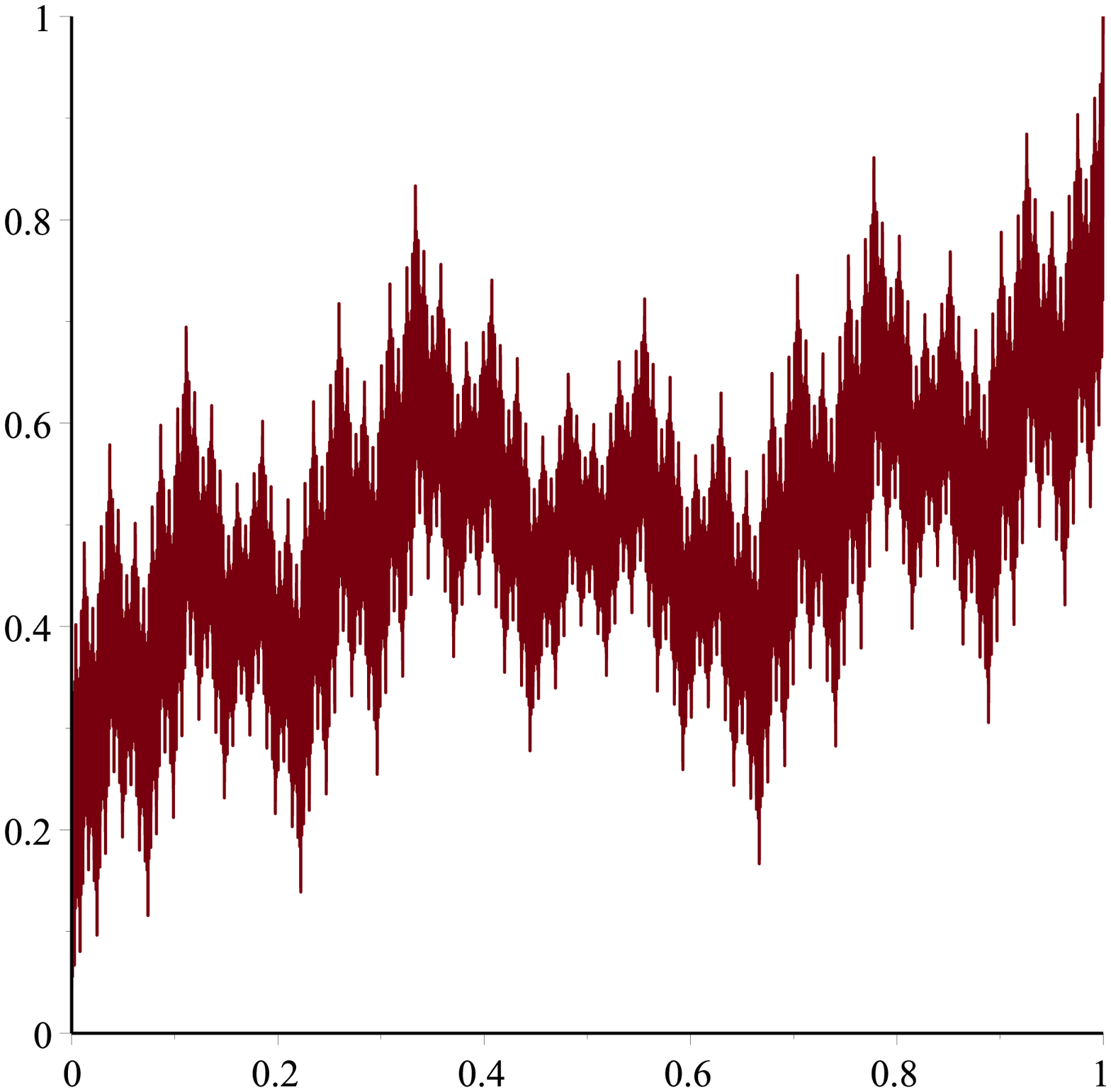, height=.25\textheight, width=.4\textwidth} \qquad\quad
\epsfig{file=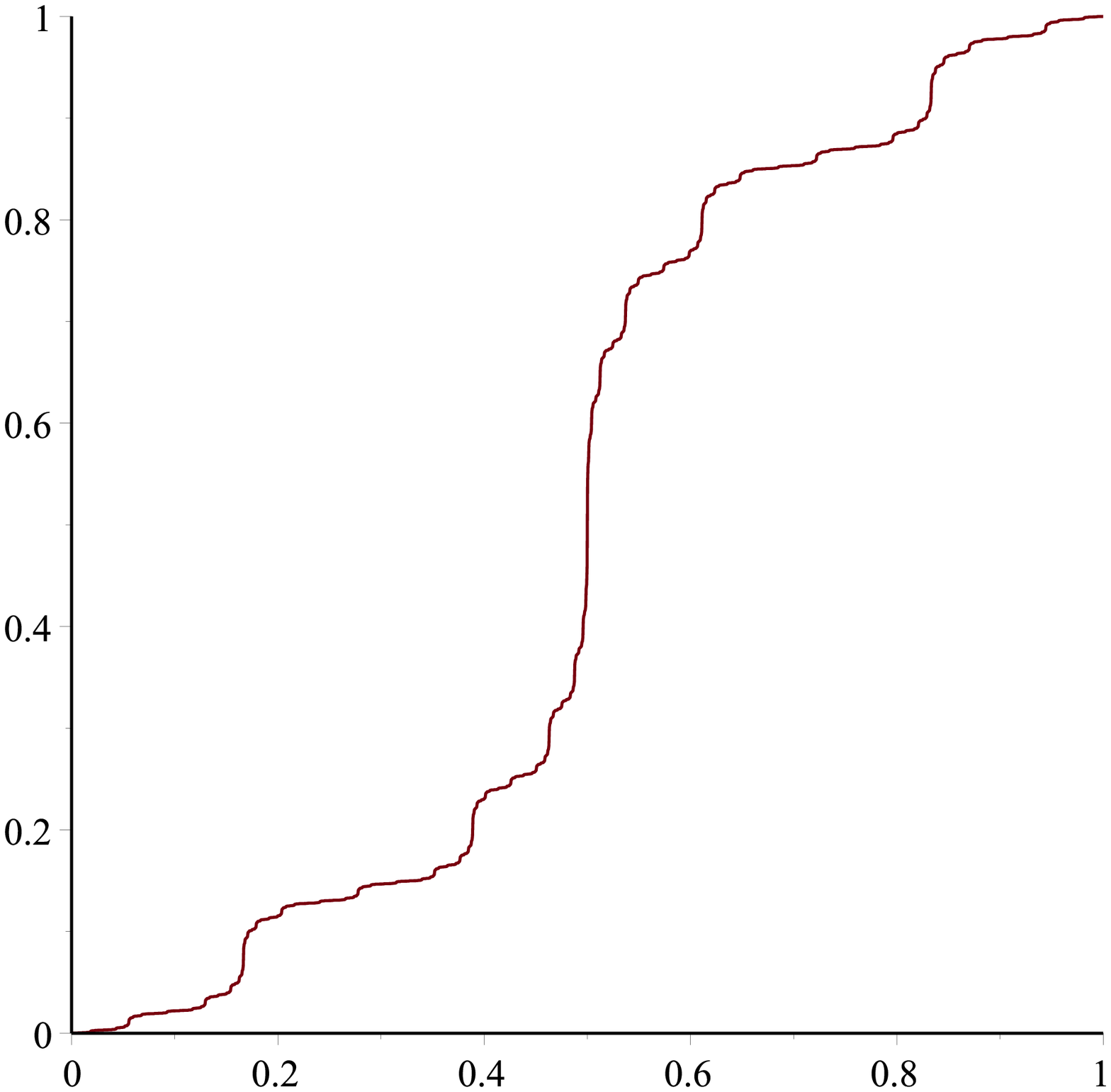, height=.25\textheight, width=.4\textwidth}
\caption{Graphs of Okamoto's function for $\alpha=5/6$ (Perkins' function; left) and $\alpha=1/5$ (right).}
\label{fig:Okamoto-graphs}
\end{center}
\end{figure}

Some basic facts on the differentiability of $F_\alpha$ are proved in \cite{okamoto}. For example, let $\alpha_0\approx 0.5598$ be the (unique) real root of $54x^3-27x^2-1=0$ in $(0,1)$. Then
\begin{itemize}
    \item If $\alpha\ge2/3$, then $F_{\alpha}$ is differentiable nowhere.
    \item If $\alpha_0\le \alpha<2/3$, then $F_{\alpha}$ is not differentiable at almost every $x\in[0,1]$, with uncountably many exceptions. (The case of $\alpha=\alpha_0$ was determined in \cite{kobayashi}).
    \item If $0<\alpha<\alpha_0$ and $\alpha\neq 1/3$, then $F_{\alpha}$ is differentiable for almost every $x\in[0,1]$, with uncountably many exceptions.
\end{itemize}
At the points where $F_{\alpha}$ is non-differentiable, it may still be the case that $F_{\alpha}$ has an infinite derivative. An extensive treatment of the set $\mathcal{D}(\alpha):=\{x\in[0,1]:F_{\alpha}'(x)=\pm\infty\}$ may be found in \cite{allaart1}.

A simple application of \cite{bedford2} gives the box-counting dimension of the graph of $F_\alpha$:
$$\dim_B \Graph{(F_\alpha)}=
\begin{cases}
 1+\frac{\log(4\alpha-1)}{\log3} & \text{ if }\alpha\ge\frac{1}{2}\\
 1 & \text{ if }\alpha<\frac{1}{2}.
\end{cases}
$$

Our goal is to prove analogous results on the differentiability and box-counting dimension of the graphs of randomized versions of these functions, which are constructed through a process having non-homogeneous contractions in the $x$-direction and random contractions in the $y$-direction.
Our construction somewhat resembles that of Dubins and Freedman \cite{dubins}, who use a probability measure on $[0,1]^2$ to create the graph of a random cumulative distribution function on $[0,1]$.
The main idea of our construction is to first partition the unit interval into $m$ pieces of (deterministic) lengths $l_0,\dots,l_{m-1}$. Then we form $m$ random affine contractions $T_0,\dots,T_{m-1}$ for which each $T_i$ maps the unit square to a rectangle having width $l_i$ and (random) height $h_i$. We iterate this process independently within each sub-rectangle, and in the limit, under very weak assumptions, obtain the graph of a continuous function. 
In Section \ref{sec:main results} we give a precise formalization these ideas. For $m=2$ our construction is contained in that of Dubins and Freedman, and the questions pursued in this article are trivial. But for $m\geq 3$ the random functions we construct are typically nowhere monotone, and both questions of differentiability and the dimension of the graph become interesting.

Graphs of random functions have of course been studied in other settings as well.
For example, Hunt \cite{hunt} calculates the almost sure Hausdorff dimension of the graphs of Weierstrass functions with random phase shifts. 
It is also well known that the graph a Brownian sample function has Hausdorff and box-counting dimension $1.5$ almost surely (see \cite[Theorem 16.4]{falconer}). 
Neither Weierstrass functions nor Brownian motion are constructed through an IFS however.

There are several ways one may add an element of randomness to the construction of box-like self-affine sets. 
For example, start with a collection $\mathbb{I}=\{\mathbb{I}_1,\dots,\mathbb{I}_n\}$, where each $\mathbb{I}_i=\{S_1^i,\dots,S_{m_i}^i\}$ is an IFS as described at the beginning of this article. At each stage of the construction, and to each rectangle in the construction at that stage, choose an IFS in $\mathbb{I}$ to apply to that rectangle.
This model was considered by Gatzouras and Lalley \cite{gatz2}, who, under restrictive assumptions on the maps in each IFS, compute the almost sure box-counting and Hausdorff dimension. 
More recently, a similar model was studied by Troscheit \cite{troscheit}, who computes the almost sure box-counting dimension under more relaxed conditions. 
The Assouad dimension in this model is considered in \cite{fraser2}.

Another example of randomizing box-like self-affine sets is to consider the translation vectors $a_i$ of the self-affine maps $\{T_i(\vec{x})=A\vec{x}+a_i\}_{i=1}^n$.
This was first considered by Falconer \cite{falconer2}, who shows the Hausdorff dimension of the attractor is the same for Lebesgue almost every choice of translation vectors $(a_1,\dots,a_n)$ provided the norm of each $A_i$ is at most $1/3$.
Solomyak \cite{solomyak} later showed that the bound $1/3$ can be replaced with $1/2$.
A further generalization of this result was given by Urbanski \cite{urbanski}, who considered countably many maps $\{T_i\}_{n\in\NN}$.
A random analog of this work is found in \cite{jordan}, where Jordan {\em et al} consider random perturbations of the translation vectors $a_i$ at each stage of the construction. 

Another model for randomizing the translation vectors at each stage is studied by J\"arvenp\"a\"a \textit{et al} \cite{jarvenpaa}. 
Here the randomization has considerable overlap with the previously discussed model, where at each stage of the construction, IFS's are randomly chosen to be applied from a finite collection.

While each of the models discussed above are very natural, they are quite different from the construction considered here.
For one thing, in our affine maps, both the vertical contraction ratios {\em and} the translation vectors are random, which is necessarily the case if one wants to obtain the graph of a continuous function. Another reason is that we do not pick our affine maps from a finite (or countable) family of IFSs, but rather choose their parameters at random from an arbitrary probability distribution on $[0,1]$.


The rest of this article is organized as follows: In Section \ref{sec:main results}, we give a formal description of our model and the precise statements of the main results. Theorem \ref{diff thm} describes a dichotomy, namely our random self-affine function is either almost surely differentiable almost everywhere, or almost surely non-differentiable almost everywhere, and the theorem specifies exactly when each case holds.
Theorem \ref{box thm} gives a formula for the almost sure box-counting dimension of the graph of a random self-affine function. At the end of Section \ref{sec:main results} we give two specific examples illustrating the theorems. 
Sections \ref{sec:differentiability} and \ref{sec:box-dimension} are dedicated to the proofs of Theorems \ref{diff thm} and \ref{box thm}, respectively.

\section{Notation and main results} \label{sec:main results}


We begin by describing the notation used for the symbolic space. First, let $0=b_0<b_1<\dots<b_{m-1}<b_m=1$ be a partition of $[0,1]$ into $m$ pieces. For each finite word of length $n$ 
$$\omega=\omega_1,\dots,\omega_n\in\mathcal{I}_n:=\{0,1,\dots,m-1\}^n,$$ 
we define
$$b_\omega:=\sum_{k=1}^nb_{\omega_k}\prod_{i=1}^{k-1}(b_{\omega_{i}+1}-b_{\omega_{i}}),$$
where the empty product will be taken to mean $1$.
We let 
$$\mathcal{I}^{*}:=\bigcup_{n=0}^\infty\mathcal{I}_n$$
denote the set of all finite words, including the empty word if $n=0$.
 The restriction of a word $\omega$ to $k\in\NN$ is denoted by $\omega|_k:=\omega_1,\dots,\omega_k\in\mathcal{I}_k$, provided $k\leq|\omega|,$ the length of $\omega$.
 Concatenation of a finite word $\omega\in\mathcal{I}^*$ with $i\in\{0,1,\dots,m-1\}$ will be denoted by $\omega i$. 
 For $\omega\in\mathcal{I}^*$, we say $\omega|_{|\omega|-1}$ is the parent of $\omega$, and $\omega i$ is a child of $\omega$.
We will denote the set of infinite words by
$$\mathcal{I}:=\{0,1,\dots,m-1\}^\NN.$$

We will now construct the graph of a random self-affine function. First fix a probability space
$(\Theta,\mathscr{F},\PP)$, and let 
\[
\{\mathbf{Y}_\omega=(y_{\omega1},\dots,y_{\omega(m-1)}):\omega\in\mathcal{I}^*\} 
\]
be a collection of independent random vectors, all having the same joint distribution, defined on $\Theta$. The base vector is $\mathbf{Y}_\varnothing:=(y_1,\dots,y_{m-1})$. For each $\omega\in \mathcal{I}^*$ we also set
\[
y_{\omega 0}:=0, \qquad y_{\omega m}:=1.
\]
Note that for fixed $\omega$, we do not require the independence of the random variables $y_{\omega 1},\dots,y_{\omega(m-1)}$. To avoid degenerate cases, we assume $\PP(y_i\in\{0,1\})=0$ for all $i=1,\dots,m-1$. This is to prevent any jump discontinuities. We also assume that
\begin{equation} \label{eq:not-diagonal}
\PP(y_{i+1}-y_i=b_{i+1}-b_i)<1, \qquad i=1,\dots,m-1.
\end{equation}
This assumption is used only in Proposition \ref{prop:easy}.

Now for each $\omega\in\mathcal{I}^*$ we define random affine transformations $T_{\omega i}: [0,1]^2\to[0,1]^2$ for $i=0,1,\dots,m-1$ by
$$T_{\omega i}:\begin{bmatrix}
x\\
y\end{bmatrix}
\mapsto
\begin{bmatrix}
b_{i+1}-b_{i} & 0\\
0 & y_{\omega (i+1)}-y_{\omega i}
\end{bmatrix}
\begin{bmatrix}
x\\
y
\end{bmatrix}+
\begin{bmatrix}
b_{i}\\
y_{\omega i}
\end{bmatrix}.
$$
Then for $n\geq 1$ and $\omega=\omega_1,\dots,\omega_n\in\mathcal{I}_n$, we define $R_{\omega}$ to be the rectangle
\[
R_\omega:=T_{\omega|_1}\circ T_{\omega|_2}\circ\dots\circ T_{\omega|_{n-1}}\circ T_\omega ([0,1]^2).
\]

Finally, let
\begin{equation}
\Gamma:=\bigcap_{n\in\NN}\bigcup_{\omega\in\mathcal{I}_n}R_\omega.   
\end{equation}
We use the following notation throughout: For $\omega\in\mathcal{I}_n\setminus\{(m-1,\dots,m-1)\}$, let $\omega'$ denote the smallest word in $\mathcal{I}_n$ (according to lexicographical order) larger than $\omega$.
If $\omega$ is maximal in $\mathcal{I}_n$, i.e. $\omega=(m-1,m-1,\dots,m-1)$, define $b_{\omega'}:=1$ and $y_{\omega'}:=1$. 
For $\omega=\omega_1,\dots,\omega_n\in\mathcal{I}_n$,
$$a_\omega:=|y_{\omega'}-y_{\omega}|$$
will denote the the ratio of the height of $R_\omega$ to the height of its parent rectangle, while
$$h_\omega:=\prod_{k=1}^{n}a_{\omega|_k}$$
will be used to denote the height of the rectangle $R_\omega$.


\begin{proposition} \label{prop:cont}
With probability one, the set $\Gamma$ is the graph of a continuous function $F:[0,1]\to[0,1]$.
\end{proposition}

\begin{proof}
We modify the argument from the proof of \cite[Theorem 4.1]{dubins}.
Each $x\in[0,1]$ admits an infinite word $\omega\in\mathcal{I}$ such that $x=\lim_{n\to\infty}b_{\omega|_n}$. If a number $x$ admits two such words, we take the one ending in all zeros. For any realization $\theta\in\Theta$ we define
\begin{equation}
\label{def of F}
  (x,F_\theta(x)):=\bigcap_{n=1}^\infty R_{\omega|_n}(\theta).  
\end{equation}

To justify that this is indeed well-defined, we need to show that for $\mathbb{P}$-almost every $\theta\in\Theta$, $\diam{R_{\omega|_n}(\theta)}\to0$ for all $\omega\in\mathcal{I}.$
Since the bases of the rectangles are deterministic, their widths clearly go to zero. Thus, it suffices to show 
\begin{equation} \label{eq:max-height}
\lim_{n\to\infty}\max\{h_{\omega}(\theta):\omega\in\mathcal{I}_n\}=0
\end{equation}
for $\mathbb{P}$-almost all $\theta$.
To that end, we note that the assumption $\mathbb{P}(y_i\in\{0,1\})=0$ implies that $\mathbb{P}(a_i<1)=1$ for all $i=0,1,\dots,m-1$. Thus, $a_i^r\to 0$ almost surely for each $i$ as $r\to \infty$, so by the Bounded Convergence Theorem there exists an $r>1$ such that 
$$\rho:=\sum_{i=0}^{m-1}\EE(a_{i}^r)<1.$$
Therefore
$$\EE\left(\sum_{\omega\in\mathcal{I}_n}h_\omega^r\right)=\rho^n.$$
Now by Markov's inequality,
$$\sum_{n=1}^\infty \mathbb{P}\left(\sum_{\omega\in\mathcal{I}_n}h_{\omega}^r\ge\eps\right)\le \sum_{n=1}^\infty\frac{1}{\eps}\EE\left(\sum_{\omega\in\mathcal{I}_n}h_{\omega}^r\right)\le \frac{1}{\eps}\sum_{n=1}^\infty \rho^n<\infty.$$
Hence by the first Borel-Cantelli lemma, the probability that $\sum_{\omega\in\mathcal{I}_n}h_{\omega}^r\ge\eps$ infinitely often is zero. Since $\eps$ was arbitrary, it follows that $\sum_{\omega\in\mathcal{I}_n}h_{\omega}^r\to 0$ almost surely, which then implies \eqref{eq:max-height} for $\mathbb{P}$-almost all $\theta$.

Now for each $\theta$ such that \eqref{eq:max-height} holds, the right hand side of \eqref{def of F} is a single point for every $\omega\in\mathcal{I}$.
Therefore $\Gamma$ is the graph of a function $F$ with probability one, and $F$ is continuous since $\Gamma$ is closed.
\end{proof}


\noindent {\bf Convention:} Throughout this paper, by a ``random self-affine function" we shall mean a random function $F$ constructed by the above process.

\begin{figure}
    \centering
    \includegraphics[scale =0.4]{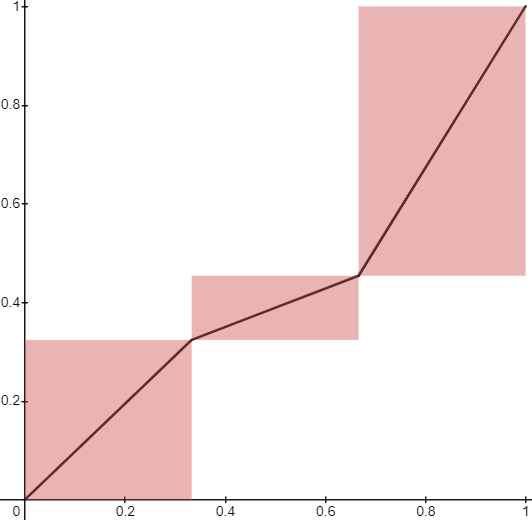}
    \caption{First iteration}
    \label{level 1}
\end{figure}

\bigskip
We illustrate our construction by taking $b_i=i/3$ for $i=0,1,2,3$. At the first stage we choose $y_1$ and $y_2$ uniformly and independently in $[0,1]$. In the first stage we get three rectangles $R_0$, $R_1,$ and $R_2$ (see Figure \ref{level 1}).
Then we repeat this process to get three sub-rectangles for each of $R_0$, $R_1$ and $R_2.$ These are seen in the left-hand graphs of Figure \ref{level 2}. Taking the images of these sub-rectangles under the appropriate maps gives the nine rectangles $R_{00},R_{01},\dots,R_{22}$ seen in the right most graph of Figure \ref{level 2}.
We show the diagonals of the rectangles in Figures \ref{level 1} and \ref{level 2} to emphasize how the graph takes shape. Figure \ref{level 3} shows the graph after five iterations.

\begin{figure}
    \begin{center}
        \includegraphics[scale=0.4]{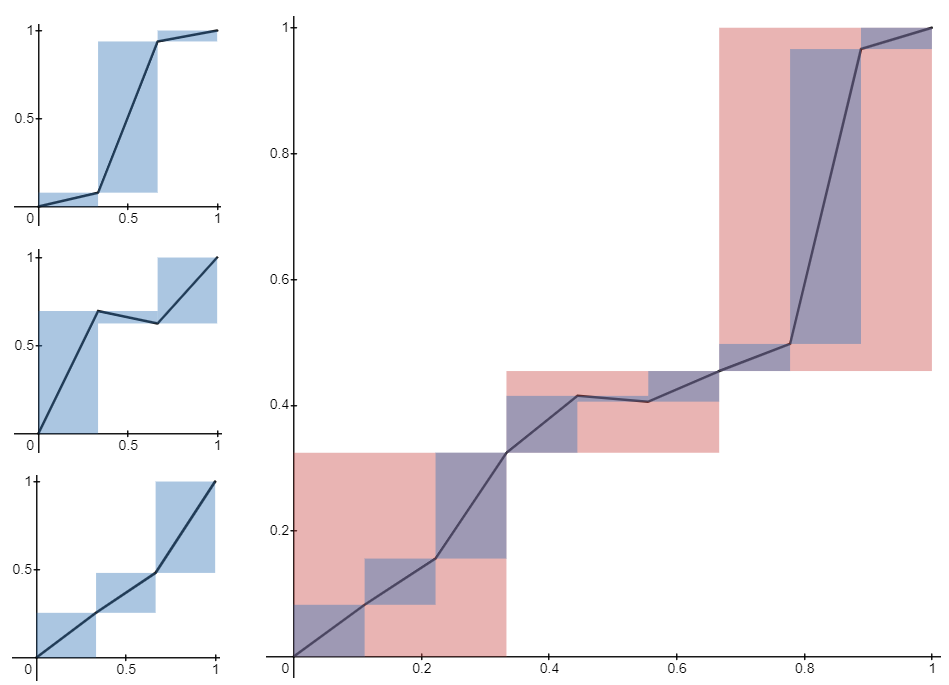}
        \caption{Second iteration}
        \label{level 2}
    \end{center}
\end{figure}

\begin{figure}
    \centering
    \includegraphics[scale =0.5]{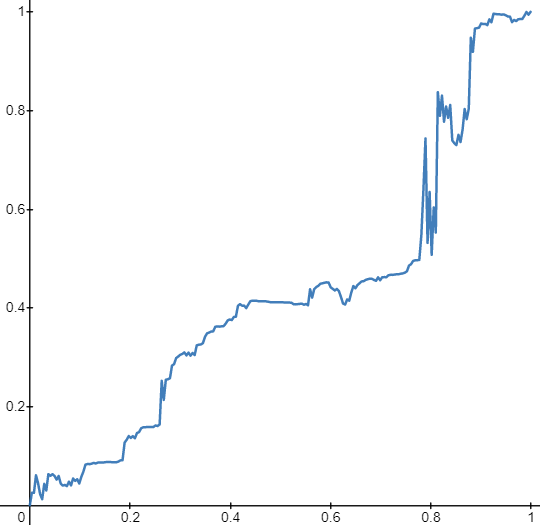}
    \caption{Fifth iteration}
    \label{level 3}
\end{figure}


For a given partition $0=b_0<b_1<\dots<b_{m-1}<b_m=1$, let $l_i:=b_{i+1}-b_i$ be the length of the interval $[b_i,b_{i+1}]$. Let $\lambda$ denote Lebesgue measure on $[0,1]$. Notice that for any finite word $\omega=(\omega_1,\dots,\omega_n)\in\mathcal{I}_n$,
$$l_\omega:=\lambda([b_\omega,b_{\omega'}])=\prod_{i=1}^nl_{\omega_i}.$$

We now state our main results, and provide examples.


\begin{theorem} \label{diff thm}
Let $F$ be a random self-affine function, and compute
$$\varphi:=\sum_{i=0}^{m-1}l_i(\EE\log a_i-\log l_i).$$
\begin{enumerate}
    \item If $\varphi<0$,
   then almost surely, $F'(x)=0$ at $\lambda$-almost every $x\in[0,1]$.
    \item If $\varphi\geq 0$,
    then almost surely, $F$ is non-differentiable at $\lambda$-almost every $x\in[0,1]$.
\end{enumerate}
\end{theorem}

We also calculate the almost sure box-counting dimension of the graph of random self-affine functions. Before stating the theorem, we recall the definition of the box-counting dimension of a set:
For $\delta>0$ and a bounded set $E\subseteq\RR^n$, let
$$N_\delta(E):=\min\{\#\mathcal{A}:\mathcal{A}\text{ is a cover of }E\text{ by cubes of side length }\delta\}.$$

\begin{definition}
Let $E\subseteq\RR^n$ be bounded. The {\em upper} and {\em lower box-counting dimensions} of $E$ are defined as
$$\overline{\dim}_B\, E:=\limsup_{\delta\to0}\frac{\log N_\delta(E)}{-\log\delta}\quad \text{ and }\quad \underline{\dim}_B\, E:=\liminf_{\delta\to0}\frac{\log N_\delta(E)}{-\log\delta},$$
respectively. If the upper and lower box-counting dimensions agree, then we say the {\em box-counting dimension} of $E$ is 
$$\dim_B E:=\lim_{\delta\to0}\frac{\log N_\delta(E)}{-\log\delta}.$$
\end{definition}

For a general overview of the basic facts about box-counting dimension, we refer the reader to \cite[Chapter 3]{falconer}. 

\begin{theorem} \label{box thm}
Let $s$ be the unique solution to the equation 
\begin{equation} \label{eq:box-dimension-equation}
1=\EE\left(\sum_{i=0}^{m-1}a_il_i^{s-1}\right).
\end{equation}
Then almost surely, $s$ is the box-counting dimension of the graph of $F$.
\end{theorem}


We end this section with two examples illustrating Theorems \ref{diff thm} and \ref{box thm}.



\begin{example}
{\rm
Let $b_0=0, b_1=\frac{2}{5}, b_2=\frac{3}{5}$, and $b_3=1$, so $l_0=l_2=\frac25$ and $l_1=\frac15$. If $y_1$ is chosen with Beta(2,1) distribution on $[0,1]$, and we set $y_2=1-y_1$, then the random function $F$ is almost surely non-differentiable almost everywhere, with box-counting dimension approximately $1.561$. The calculations are as follows:

First, $y_1$ has density function $2x$ for $0\le x\le1$. Therefore, since $a_0=y_1$, $a_2=1-y_2=y_1$, and $a_1=|y_2-y_1|=|1-2y_1|$, we have
$$\EE(\log a_0)=\EE(\log a_2)=2\int_0^1 x\log x dx=-\frac{1}{2}$$
and
$$\EE(\log a_1)=2\int_0^1 x\log|1-2x|dx=-1.$$
Then compute $\varphi=\sum_{i=0}^{m-1}l_i(\EE\log a_i-\log l_i)\approx 0.455>0$.
So by Theorem \ref{diff thm}, $F$ is almost surely non-differentiable at almost every $x\in(0,1)$.

In order to apply Theorem \ref{box thm}, we first calculate $\EE(a_0)=\EE(a_2)=2/3$ and 
$$\EE(a_1)=\EE|1-2y_1|=2\int_0^1 x|1-2x|dx=\frac{1}{2}.$$
Solving \eqref{eq:box-dimension-equation} for $s$ gives $\dim_B \mathrm{Graph}(F)=s\approx1.561$.
}
\end{example}

Note that it is usually not possible to give a closed form expression for the box-counting dimension. An exception is when the $x$-axis intervals are all the same length, as shown below.

\begin{example}
{\rm
If $b_i=i/m$ for $i=0,1,\dots,m$, and $y_1,\dots,y_{m-1}$ are chosen independently according to the uniform distribution on $[0,1]$, then $F$ is almost surely differentiable almost everywhere and
\begin{equation} \label{eq:evenly-spaced-box-dimension}
\dim_B \Graph(F)=1+\frac{\log(m+1)-\log 3}{\log m}
\end{equation}
almost surely. The calculations are as follows: First, $\EE(a_0)=\EE(a_{m-1})=\frac{1}{2}$, while
$$\EE(a_i)=\int_0^1\int_0^1|x-y|dxdy=\frac{1}{3}$$ for all $1\le i\le m-2$. Thus,
    $$\EE\left(\sum_{i=0}^{m-1}a_i l_i^{s-1}\right)=m^{1-s}\Big(\frac{m+1}{3}\Big).$$
Applying Theorem \ref{box thm} yields \eqref{eq:evenly-spaced-box-dimension}.
For the differentiablity of $F$, we have that 
$$\varphi=\sum_{i=0}^{m-1}l_i(\EE \log a_i-\log l_i)=\frac{2\log m-3m+2}{2m}<0 \quad\forall m\in\NN.$$
By Theorem \ref{diff thm}, $F$ is almost surely differentiable almost everywhere.
}
\end{example}

\section{Proof of Theorem \ref{diff thm}} \label{sec:differentiability}


Throughout, let $F$ be a random self-affine function as constructed in the previous section. The following is an immediate consequence of the construction:

\begin{lemma} \label{initial set up}
For every finite subset $A\subseteq\mathcal{I}^*$, the random variables $\{a_\omega:\omega\in A\}$ are mutually independent provided no two words in $A$ share the same parent. Furthermore, if $\omega=(\omega_1,\dots,\omega_n)\in\mathcal{I}^*$, then $a_\omega$ has the same distribution as $a_{\omega_n}$. Consequently
$$\EE(a_\omega)=\EE(a_{\omega_n}).$$
\end{lemma}

We first prove that the only possible finite derivative of $F$ is zero:


\begin{proposition} \label{prop:easy}
If $F$ is differentiable at $x$, then almost surely $F'(x)=0.$
\end{proposition}

\begin{proof}
Suppose $F$ is differentiable at $x$, and let $\omega\in\mathcal{I}$ be the coding of $x$. Letting $x_n=b_{\omega|_n}$ and $x_n'=b_{(\omega|_n)'}$, we then have
 \begin{equation*}
     F'(x) =\lim_{n\to\infty}\frac{F(x_n')-F(x_n)}{x_n'-x_n} 
     =\lim_{n\to\infty} \prod_{i=1}^n\frac{\tilde{a}_{\omega|_i}}{l_{\omega_i}},
 \end{equation*}
where we write $\tilde{a}_{\omega|_i}:=y_{(\omega|_i)'}-y_{\omega|_i}$, so $a_{\omega|_i}=|\tilde{a}_{\omega|_i}|$.
 If this last limit is finite and nonzero, then $\lim_{n\to\infty}\frac{\tilde{a}_{\omega|_n}}{l_{\omega_n}}=1.$ This means for every $\eps>0$, there exists an $N$ such that $n\ge N$ implies $l_{\omega_{n}}(1-\eps)\le \tilde{a}_{\omega|_n}\le l_{\omega_{n}}(1+\eps)$.
Fix $\eps>0$ small enough so that $\PP\big(l_{i}(1-\eps)\le \tilde{a}_{i}\le l_{i}(1+\eps)\big)<1$ for $i=0,\dots,m-1$. This can be done since every $\tilde{a}_i$ is not identically $l_i$ in view of the assumption \eqref{eq:not-diagonal}.
Put
\[
p:=\max\big\{\PP\big(l_{i}(1-\eps)\le \tilde{a}_{i}\le l_{i}(1+\eps)\big): i=0,\dots,m-1\big\}. 
\]
Then by identical distribution
 $$\PP\big(l_{\omega_{n}}(1-\eps)\le \tilde{a}_{\omega|_n}\le l_{\omega_{n}}(1+\eps)\big)\le p<1$$
 for all $n\in\NN.$
 Since the random variables $\tilde{a}_{\omega|_i}$, $i\in\NN$ are independent, we have that 
 \begin{equation*}
     P\big(F'(x)\ \mbox{exists and}\ \ne0\big) \le P\left(\lim_{n\to\infty}\frac{\tilde{a}_{\omega|_n}}{l_{\omega_n}}=1\right)
      \le \lim_{n\to\infty}p^n = 0,
 \end{equation*}
completing the proof.
\end{proof}

\begin{remark}
{\rm
By using a stochastic version of the careful argument in the proof of \cite[Proposition 2.1]{allaart2}, the conclusion of Proposition \ref{prop:easy} may be obtained under the weaker assumption that $\PP(y_i=b_i)<1$ for at least one index $i\in\{1,\dots,m-1\}$, instead of \eqref{eq:not-diagonal}. This involves additional technicalities, however.
}
\end{remark}

We now begin work on proving Theorem \ref{diff thm}.
Note that the lengths $(l_0,\dots,l_{m-1})$ form a probability vector $\Vec{p}$. As in Proposition \ref{prop:cont}, each $x\in[0,1]$ may be identified by an infinite word $\omega(x)\in\mathcal{I}$. Let $C_n^{(i)}(x)$ count the number of times $i$ appears in the first $n$ letters of $\omega(x)$, that is
$$C_n^{(i)}(x):=\#\{k\le n:\omega(x)_k=i\}.$$
We say that $x$ is $\Vec{p}$-$normal$ if for all $i=0,1,\dots,m-1$
\begin{equation} \label{eq:p-normal}
\lim_{n\to\infty}\frac{C_n^{(i)}(x)}{n}=l_i.
\end{equation}


\begin{lemma} \label{lem:almost-all-normal}
Pick an $x\in[0,1]$ randomly according to $\lambda$. Then 
$$\lambda(\{x:x\text{ is }\Vec{p}\text{-normal}\})=1.$$
\end{lemma}

\begin{proof}
Let $\omega\in\mathcal{I}$ be the coding for $x$. Fix an $i\in\{0,1,\dots,m-1\}$ and consider the random variables
$$X_k^{(i)}=\begin{cases}
1 & \text{ if }\omega_k=i\\
0 & \text{ otherwise}.
\end{cases}
$$
A straightforward exercise shows that $\{X_k^{(i)}\}_{k\in\NN}$ is a sequence of independent identically distributed random variables with mean $l_i$ with respect to $\lambda$. By the strong law of large numbers,
$$\frac{C_n^{(i)}(x)}{n}=\frac{1}{n}\sum_{k=1}^nX_k^{(i)}\xrightarrow[]{a.s.}l_i.$$
But as this holds for all $i$, $x$ is almost surely $\Vec{p}$-normal.
\end{proof}

Next, for $x\in[0,1]$, define 
$$M_n(x)=\inf\{k>n:\omega(x)_k=\omega(x)_n\}.$$
Basically, $M_n(x)$ records the very next time we see the letter $\omega(x)_n$ again in the coding for $x$. We will make use of the following lemma, whose proof is a slight modification of an argument given in \cite{lax}.


\begin{lemma} \label{lax lemma}
If $x$ is $\Vec{p}$-normal, then 
$$M_n(x)=n+o(n) \qquad\mbox{as $n\to\infty$}.$$
\end{lemma}

\begin{proof}
Let $\omega\in\mathcal{I}$ be the coding for $x$. Take $i\in\mathcal{I}_1$ and let $n_k$ be the position of the $k$th occurrence of the digit $i$ in $\omega$. Write $M_{n_k}(x)=M_{n_k}$ and $C_{n_k}^{(i)}(x)=C_{n_k}^{(i)}$ for short.  
Then $C_{M_{n_k}}^{(i)}=C_{n_k}^{(i)}+1,$ since $M_n(x)$ denotes the first reappearance of $\omega_n$. This means
\begin{equation*}
    l_i = \lim_{k\to\infty}\frac{C_{n_k}^{(i)}+1}{n_k}
    = \lim_{k\to\infty}\frac{C_{M_{n_k}}^{(i)}}{M_{n_k}}\frac{M_{n_k}}{n_k}
    = l_i\lim_{k\to\infty}\frac{M_{n_k}}{n_k},
\end{equation*}
which in turn implies $M_{n_k}/n_k\to1$. However, this was true for each $i=0,1,\dots,m-1$, and therefore $M_n/n\to1$ as well.
\end{proof}

If $\omega(x)$ admits a long string of $0$'s  or $(m-1)$'s, then the smallest basic interval containing both $x$ and $x+\delta$ can be quite large, even when $\delta$ is very small. The next lemma deals with this issue.


\begin{lemma} \label{p normal}
Let $x$ be $\Vec{p}$-normal with coding $\omega\in\mathcal{I}$. For $|\delta|<1$ let $n(\delta)$ be the largest $n$ such that both $x$ and $x+\delta$ are contained in the basic interval $[b_{\omega|_n},b_{(\omega|_n)'}]$ of length $l_{\omega|_n}$. Note that $n(\delta)\to\infty$ as $|\delta|\to0$. Then
$$\lim_{\delta\to 0}\Bigg(\frac{l_{\omega|_{n(\delta)}}}{|\delta|}\Bigg)^{\frac{1}{n(\delta)}}=1.$$
\end{lemma}

\begin{proof}
Let $L:=\min\{l_0,\dots,l_{m-1}\}$. We first show that
\begin{equation} \label{eq:technical-inequalities}
L^{M_{n(\delta)+1}-n(\delta)}l_{\omega|_{n(\delta)}}\le|\delta| \le l_{\omega|_{n(\delta)}},
\end{equation}
where again we write $M_n(x)=M_n$ for convenience. The second inequality is immediate from the definition of $n(\delta)$; the first is more involved.
Let $\gamma\in\mathcal{I}$ be the coding for $x+\delta$. Then $\gamma=\omega_1\dots\omega_{n(\delta)}\gamma_{n(\delta)+1}\dots$, where $\gamma_{n(\delta)+1}\ne\omega_{n(\delta)+1}$. 

First assume $\delta>0$, then since $x+\delta>x$, $\gamma>_{\text{lex}}\omega$ in the lexicographic order, implying $\omega_{n(\delta)+1}< m-1$. Since $x$ and $x+\delta$ have positive distance, there must be a basic interval between them. As discussed, the worst case scenario is when $\omega$ admits a long run of $(m-1)$'s after $\omega_{n(\delta)+1}$, while $\gamma$ has a corresponding run of $0$'s. In this situation:
$$
\begin{matrix}
\omega= & \omega_1 & \dots & \omega_{n(\delta)} & \omega_{n(\delta)+1} & m-1 & \dots& m-1 & \omega_N &\dots\\
\gamma = & \omega_1 &\dots & \omega_{n(\delta)} & \omega_{n(\delta)+1}+1 & 0 & \dots & 0 & \gamma_N & \dots,
\end{matrix}
$$
where $\omega_N\ne m-1$ marks the end of the run of $(m-1)$'s, the length of such run being $N-(n(\delta)+2)\le M_{n(\delta)+1}-n(\delta)$. Therefore
$$\delta\ge L^{M_{n(\delta)+1}-n(\delta)}l_{\omega|_n(\delta)},$$
as claimed. The argument for $\delta<0$ is symmetric. Finally, from \eqref{eq:technical-inequalities} we conclude
$$1\le\Bigg(\frac{l_{\omega|_{n(\delta)}}}{|\delta|}\Bigg)^{\frac{1}{n(\delta)}}\le \Big(L^{-M_{n(\delta)+1}+n(\delta)}\Big)^{\frac{1}{n(\delta)}}\to 1\quad\mbox{as $\delta\to 0$}$$
by Lemma \ref{lax lemma}.
\end{proof}

We are now ready to prove Theorem \ref{diff thm}.


\begin{proof}[Proof of Theorem \ref{diff thm}] 
For part (1), assume that
$$\sum_{i=0}^{m-1}l_i\left(\EE\log a_i-\log l_i\right)<0,$$
and let $x$ be $\Vec{p}$-normal coded by $\omega\in\mathcal{I}$. For small $\delta$ we write $n=n(\delta)$ (where $n(\delta)$ was defined in Lemma \ref{p normal}) and consider the $n^{\text{th}}$ root of the difference quotient

$$\Big|\frac{F(x+\delta)-F(x)}{\delta}\Big|^{\frac{1}{n}}.$$
By Lemma \ref{p normal}, $l_{\omega|_n}$ will decrease at the same rate as $\delta$ when $\delta\searrow0$. Observe that
\begin{align*}
    \log\Big|\frac{F(x+\delta)-F(x)}{l_{\omega|_n}}\Big|^{\frac{1}{n}} & \le \frac{1}{n}\log\frac{h_{\omega|_n}}{l_{\omega|_n}}
 = \frac{1}{n}\sum_{k=1}^n\log a_{\omega|_k}-\frac{1}{n}\sum_{k=1}^n\log l_{\omega_k}\\
    & = \frac{1}{n}\sum_{i=0}^{m-1}\left(\sum_{k:\,\omega_k=i}\log a_{\omega|_k}-\sum_{k:\,\omega_k=i}\log l_{\omega_k}\right)\\
    & = \sum_{i=0}^{m-1}\frac{C_n^{(i)}(x)}{n}\left(\sum_{k:\,\omega_k=i}\frac{\log a_{\omega|_k}}{C_n^{(i)}(x)}-\log l_{i}\right).
\end{align*}
Now by independence of the $a_{\omega|_k}$'s and the strong law of large numbers,
\[
\sum_{k:\,\omega_k=i}\frac{\log a_{\omega|_k}}{C_n^{(i)}(x)}\to \EE\log a_i\quad \mbox{a.s.}
\]
Hence using \eqref{eq:p-normal} we get that almost surely
$$\lim_{\delta\to0}\log\Big|\frac{F(x+\delta)-F(x)}{\delta}\Big|^{\frac{1}{n(\delta)}}\le\sum_{i=0}^{m-1}l_i\left(\EE\log a_i-\log l_i\right)<0.$$
But this implies 
$$\lim_{\delta\to0}\Big|\frac{F(x+\delta)-F(x)}{\delta}\Big|=0 \qquad\mbox{a.s.}$$
Since Lebesgue-almost every $x$ is $\Vec{p}$-normal by Lemma \ref{lem:almost-all-normal}, an application of Fubini's theorem yields that with probability one, $F'(x)=0$ almost everywhere.

\bigskip
For part (2), assume first that 
 $$\sum_{i=0}^{m-1}l_i(\EE\log a_i-\log l_i)>0.$$
 Let $x$ be $\Vec{p}$-normal, coded by $\omega\in\mathcal{I}$, and let $x_n=b_{\omega|_n}$, $x_n'=b_{(\omega|_n)'}$. It is easy to see $F$ is not differentiable at $x$ by looking at the sequence of secant lines. In the same way as above we get
\begin{align*}
     \frac{1}{n}\log\frac{|F(x_n')-F(x_n)|}{x_n'-x_n} & =  \sum_{i=0}^{m-1}\frac{C_n^{(i)}(x)}{n}\left(\sum_{k:\omega_k=i}\frac{\log a_{\omega|_k}}{C_n^{(i)}(x)}-\log l_{i}\right)\\
     & \xrightarrow[]{n\to\infty} \sum_{i=0}^{m-1}l_i\big(\EE\log\:a_i-\log l_i\big)>0
\end{align*}
almost surely. But this implies that the difference quotients, in absolute value, increase exponentially. Therefore $F$  is not differentiable at $x$. 
 
The more difficult case is when we assume 
$$\sum_{i=0}^{m-1}l_i(\EE\log a_i-\log l_i)=0.$$
The next lemma will be needed to deal with this case; see \cite[Exercises 4.1.8-4.1.11]{durrett}. 

 
 
\begin{lemma} \label{oscillation lemma}
Let $\{X_n\}_{n\in\NN}$ be independent and identically distributed non-degenerate random variables with mean $0$. Let $S_n=X_1+\dots+X_n$. Then almost surely $S_n>0$ infinitely often and $S_n<0$ infinitely often.
\end{lemma}

We now complete the proof of Theorem \ref{diff thm}.
Recall that we constructed the random function $F$ on the sample space $\Theta$; we will denote by $F_\theta$ a realization of a random self-affine function for $\theta\in\Theta$.
 We consider the product space $\Theta\times[0,1]$ along with the product $\sigma$-algebra $\mathscr{F}\times\mathscr{B}$, where $\mathscr{B}$ denotes the Borel $\sigma$-algebra on $[0,1]$, and product measure $\tilde{\mathbb{P}}:=\PP\times\lambda$. Define random variables
 $$X_k(\theta,x):=\log\frac{a_{\omega(x)|_k}(\theta)}{l_{\omega_k(x)}},$$
where $\omega(x)$ is the coding for $x\in[0,1]$, and $a_{\omega(x)|_k}(\theta)$ is the ratio $a_{\omega|_k}$ for a given $x$ and realization $F_\theta$. Note that the $X_k$'s are mutually independent with respect to $\tilde{\mathbb{P}}$. Letting $\tilde{\EE}$ denote expectation with respect to the product measure $\tilde{\mathbb{P}}$, we find
 \begin{align*}
     \Tilde{\EE}(X_k) & = \int_{[0,1]}\int_{\Theta}\log\frac{a_{\omega(x)|_k}(\theta)}{l_{\omega_k(x)}}d\PP(\theta)dx\\
     & = \int_0^1\EE\Big(\log\frac{a_{\omega(x)|_k}}{l_{\omega_k(x)}}\Big)dx\\
     & = \sum_{i=0}^{m-1}\int_{b_i}^{b_{i+1}}\EE\Big(\log\frac{a_{\omega(x)|_k}}{l_{\omega_k(x)}}\Big)dx\\
     & = \sum_{i=0}^{m-1}l_i\EE\Big(\log\frac{a_i}{l_i}\Big)=0.
 \end{align*}
 Therefore, by Lemma \ref{oscillation lemma} we may (almost surely with respect to $\tilde{\mathbb{P}}$) find a sub-sequence $(k_j)$ such that  $S_{k_j}>0$ for each $j$, where $S_n:=X_1+\dots+X_n$. To conclude we observe
 \begin{equation*}
    \log\Big|\frac{F_\theta(x_{k_j}')-F_\theta(x_{k_j})}{x_{k_j}'-x_{k_j}}\Big|  = \log\prod_{i=1}^{k_j}\frac{a_{\omega(x)|_i}(\theta)}{l_{\omega_{i}(x)}}
     = \sum_{i=1}^{k_j}\log\frac{a_{\omega(x)|_i}(\theta)}{l_{\omega_{i}(x)}}
     = S_{k_j}>0. 
 \end{equation*}
 So for almost every $\theta\in\Theta$ and almost every $x\in[0,1]$, $F_\theta$ is non-differentiable at $x$, in view of  Proposition \ref{prop:easy}.
 \end{proof}




\section{Proof of Theorem \ref{box thm}} \label{sec:box-dimension}


Let $G$ denote the graph of a random self-affine function and let $\delta:=\min\{l_0,\dots,l_{m-1}\}$. To estimate $N_{\delta^n}(G)$, we construct a stopping set $Q_n$ in the following way:
For each infinite word $\omega\in\mathcal{I}$, let $k(\omega)$ be the unique integer $k$ such that 
$$\delta^{n+1}<l_{\omega_1}\cdots l_{\omega_k}\le \delta^n< l_{\omega_1}\cdots l_{\omega_{k-1}}.$$
We let $Q_n$ denote the finite set of finite words obtained by restricting each $\omega$ to $\omega|_{k(\omega)}$. Note that each word in $Q_n$ has length at least $n$.  The point is that for $\omega\in Q_n$, the width $l_\omega$ of the rectangle $R_\omega$ is comparable to $\delta^n$. Thus, a reasonable approximation for the number of boxes of side length $\delta^n$ needed to cover the graph is
\[
\delta^{-n}\sum_{\omega\in Q_n} h_\omega,
\]
since we cannot do much better than to cover each rectangle $R_\omega$ completely. However, this last sum is awkward to work with, so instead we follow Troscheit \cite{troscheit} and consider the sums
$$X_n:=\sum_{\omega\in Q_n}h_\omega l_\omega^{s-1},$$
where $s$ is as in the statement of Theorem \ref{diff thm}. Asymptotically, the process $\{X_n\}$ (suitably normalized) provides a good estimate of $N_{\delta^n}(G)$, but it is not a martingale. Therefore we consider the conditional expectation
$$Y_n:=\EE(X_n\big|\mathscr{F}_n),$$
where $\mathscr{F}_n:=\sigma\{\mathbf{Y}_{\omega}:\omega\in \bigcup_{k=0}^{n-1}\mathcal{I}_k\}$ is the sigma algebra generated by the first $n$ levels of the construction of $F$. It turns out that $\{Y_n\}$ is a martingale and a good approximation of $\{X_n\}$ as well. The following lemmas are to prove these assertions.


Let $(p_0,\dots,p_{m-1})$ be the probability vector defined by
\[
p_k:=\EE(a_k l_k^{s-1}), \qquad k=0,1,\dots,m-1.
\]
Note that $\sum_{k=0}^{m-1}p_k=1$ by the choice of $s$.
The first lemma is standard; however we include a proof for completeness.

\begin{lemma} \label{key identity}
Let $Q$ be a finite subset of $\mathcal{I}^*$ such that for all $\omega\in\mathcal{I}$ there exists a unique $k\in\NN$ such that $\omega|_k\in Q$. Then 
\begin{equation}
\label{key id equation}
   \sum_{\omega\in Q}\prod_{i=1}^{|\omega|}p_{\omega_i}=1. 
\end{equation}
\end{lemma}

\begin{proof}
 We induct on $N(Q):=\max\{|\gamma|:\gamma\in Q\}$. If $N(Q)=1$, then $Q=\mathcal{I}_1$ and
 $$\sum_{\omega\in Q}\prod_{i=1}^{|\omega|}p_{\omega_i}=\sum_{k=0}^{m-1}p_k=1.$$
Now suppose that $N(Q)=n$ and (\ref{key id equation}) holds for any stopping set $Q'$ with $N(Q')\leq n-1$. Define the sets $Q^{\max}:=\{\gamma\in Q:|\gamma|=n\}$, $Q':=\{\omega|_{n-1}:\omega\in Q^{\max}\}$, and $Q'':=(Q\setminus Q^{\max})\cup Q'$. Observe that $Q'$ and $Q\setminus Q^{\max}$ are disjoint. Then we have
 \begin{align*}
     \sum_{\omega\in Q}\prod_{i=1}^{|\omega|}p_{\omega_i} & =
     \sum_{\omega\in Q\setminus Q^{\max}}\prod_{i=1}^{|\omega|}p_{\omega_i}+\sum_{\omega\in Q^{\max}}\prod_{i=1}^{n}p_{\omega_i}\\
     & = \sum_{\omega\in Q\setminus Q^{\max}}\prod_{i=1}^{|\omega|}p_{\omega_i} +\sum_{\omega\in Q'}\prod_{i=1}^{n-1}p_{\omega_i}\sum_{k=0}^{m-1}p_k\\
     & =  \sum_{\omega\in Q\setminus Q^{\max}}\prod_{i=1}^{|\omega|}p_{\omega_i} +\sum_{\omega\in Q'}\prod_{i=1}^{n-1}p_{\omega_i}\\
     & = \sum_{\omega\in Q''}\prod_{i=1}^{|\omega|}p_{\omega_i}=1
 \end{align*}
by the induction hypothesis, since $Q''$ is a stopping set with $N(Q'')=n-1$. 
\end{proof}


For each $n\in\NN$ and $\omega\in \mathcal{I}_n$, we define a sub-stopping set 
 $$Q_{n}(\omega):=\{\gamma\in Q_n:\gamma|_n=\omega\}.$$ 
For $\omega\in \mathcal{I}^*$ we will use the short hand notation
\[
t_\omega:=h_\omega l_\omega^{s-1},
\]
so $X_n=\sum_{\omega\in Q_n} t_\omega$. For $\omega\in Q_n$ and $\gamma\in Q_{n}(\omega)$, we define
$$\tilde{t}_\gamma:=\frac{t_\gamma}{t_\omega}=\prod_{k=n+1}^{|\gamma|}a_{\gamma|_k}l_{\gamma_k}^{s-1}.$$

\begin{corollary} \label{key id 2}
For all $n\in\NN$ we have
$$Y_n=\sum_{\omega\in\mathcal{I}_n}t_\omega.$$
\end{corollary}

\begin{proof}
  Observe that
 \begin{equation}
   \begin{split}  \label{Yn formula}
     Y_n & = \EE(X_n\big|\mathscr{F}_n) = \EE\left[\sum_{\omega\in Q_n}t_\omega \bigg|\mathscr{F}_n\right]\\
     & = \EE\left[\sum_{\omega\in \mathcal{I}_{n}}t_\omega \sum_{\gamma\in Q_{n}(\omega)}\tilde{t}_{\gamma}\bigg|\mathscr{F}_n\right]
      = \sum_{\omega\in \mathcal{I}_{n}}t_\omega \EE\left(\sum_{\gamma\in Q_{n}(\omega)}\tilde{t}_{\gamma}\right),
     \end{split}
 \end{equation}
where the last equality uses that $t_\omega$ is $\mathscr{F}_n$-measurable and $\tilde{t}_{\gamma}$ is independent of $\mathscr{F}_n$.
 To evaluate the last expectation in \eqref{Yn formula}, we note that for any $\gamma\in Q_n(\omega)$
 $$\tilde{t}_\gamma=\prod_{k=n+1}^{|\gamma|}a_{\gamma|_k}l_{\gamma_k}^{s-1}.$$
 By letting $Q'=\{\gamma'=(\gamma_{n+1},\dots,\gamma_{|\gamma|}):\gamma\in Q_n(\omega)\}$ we can rewrite (\ref{Yn formula}) as
 \begin{equation*}
     Y_n = \sum_{\omega\in \mathcal{I}_{n}}t_\omega \sum_{\gamma'\in Q'}\prod_{i=1}^{|\gamma'|}\EE(a_{\gamma_i'})l_{\gamma_i'}^{s-1}
       = \sum_{\omega\in \mathcal{I}_{n}}t_\omega,
 \end{equation*}
 where the last equality follows from Lemma \ref{key identity}.
\end{proof}


\begin{lemma} \label{martingale}
$\{Y_n\}_{n\in\NN}$ is a martingale with respect to the filtration $\{\mathscr{F}_n\}$.
\end{lemma}

\begin{proof}
Indeed,
\begin{align*}
    \mathbb{E}(Y_{n+1}\big|\mathscr{F}_n) & = \mathbb{E}\left[\sum_{\omega\in \mathcal{I}_{n+1}}t_\omega \bigg|\mathscr{F}_n\right]
     = \mathbb{E}\left[\sum_{\omega\in \mathcal{I}_{n}}t_\omega \sum_{i=0}^{m-1}a_{\omega i}l_i^{s-1}\bigg|\mathscr{F}_n\right]\\
    & = \sum_{\omega\in \mathcal{I}_{n}}t_\omega \mathbb{E}\left(\sum_{i=0}^{m-1}a_il_i^{s-1}\right)=Y_n.\\
\end{align*}
The first equality is from Corollary \ref{key id 2}. The second equality uses that for $\omega\in\mathcal{I}_n$, $a_{\omega i}$ is independent of $\mathscr{F}_n$ and has the same distribution as $a_i$; and the last equality follows from the definition of $s$. Therefore, $\{Y_n\}_{n\in\NN}$ is a martingale.
\end{proof}

In order to make stronger conclusions about the convergence of $Y_n$, we will show that the process $\{Y_n\}$ is $\mathcal{L}^2$-bounded. In the next two lemmas we make use of the quantity
\[
\alpha:=\EE\left(\sum_{i=0}^{m-1}a_i^2 l_{i}^{2(s-1)}\right).
\]
Observe that $\alpha<1$ by the definition of $s$.


\begin{lemma} \label{L2 bounded} We have
$$\sup_{n\in\NN}\EE(Y_n^2)<\infty.$$
\end{lemma}

\begin{proof}
To begin, we look at the conditional expectation of $Y_{n+1}^2$:
 \begin{equation}
 \begin{split}
 \label{conditional yn+1}
    \mathbb{E}(Y_{n+1}^2 \big|\mathscr{F}_n)& =\mathbb{E}\left[\left(\sum_{\omega\in \mathcal{I}_{n}}t_\omega \sum_{i=0}^{m-1}a_{\omega i}l_i^{s-1}\right)^2\Bigg| \mathscr{F}_n\right]\\
    & = \mathbb{E}\left[\sum_{\omega\in \mathcal{I}_{n}}t_\omega^2 \left(\sum_{i=0}^{m-1}a_{\omega i}l_i^{s-1}\right)^2 \right.\\
    &\qquad\quad \left. +\sum_{\sigma\ne\tau}t_\sigma t_\tau\left(\sum_{i=0}^{m-1}a_{\sigma i}l_i^{s-1}\right) \left(\sum_{i=0}^{m-1}a_{\tau i}l_i^{s-1}\right)\:\Bigg|\:\mathscr{F}_n\right]\\
\end{split}    
\end{equation}
where in the last line we sum over the pairs $(\sigma,\tau)$ of different words in $\mathcal{I}_n$. The conditional expectation in the second term can be simplified:

\begin{align*}
   \EE\Bigg[\sum_{\sigma\ne\tau} & t_\sigma t_\tau\Big(\sum_{i=0}^{m-1}a_{\sigma i}l_i^{s-1}\Big) \Big(\sum_{i=0}^{m-1}a_{\tau i}l_i^{s-1}\Big)\:\Big|\:\mathscr{F}_n\Bigg] \\
    & = \sum_{\sigma\ne\tau}t_\sigma t_\tau \EE\Bigg[\Big(\sum_{i=0}^{m-1}a_{\sigma i}l_i^{s-1}\Big)\Big(\sum_{i=0}^{m-1}a_{\tau i}l_i^{s-1}\Big)\:\Big|\:\mathscr{F}_n\Bigg]\\
    & = \sum_{\sigma\ne\tau}t_\sigma t_\tau \EE\left(\sum_{i=0}^{m-1}a_{ i}l_i^{s-1}\right)\EE\left(\sum_{i=0}^{m-1}a_{i}l_i^{s-1}\right)\\
    & = \sum_{\sigma\ne\tau}t_\sigma t_\tau \le Y_n^2.\\
\end{align*}
Here the second equality follows since $a_{\sigma i}$ and $a_{\tau i}$ are independent of $\mathscr{F}_n$ and each other for $\sigma\ne\tau$.
Substituting this back in (\ref{conditional yn+1}) we obtain
\begin{equation*}
 \mathbb{E}(Y_{n+1}^2\big|\mathscr{F}_n) \le
  Y_n^2 + \sum_{\omega\in \mathcal{I}_{n}}t_\omega^2 \mathbb{E}\Bigg[\Big(\sum_{i=0}^{m-1}a_{i}l_i^{s-1}\Big)^2\Bigg]
      =: Y_n^2 + C\sum_{\omega\in \mathcal{I}_{n}}t_\omega^2.
\end{equation*}
Taking expectation of both sides yields
\begin{equation*}
    \EE(Y_{n+1}^2) \le  \EE(Y_n^2)+C\EE\left(\sum_{\omega\in \mathcal{I}_{n}}t_\omega^2 \right)
  = \EE(Y_n^2)+C\alpha^n.
\end{equation*}
Since $\alpha<1$, it follows that
$$\EE(Y_{n}^2)\le \EE(Y_0^2)+C\sum_{k=1}^{\infty}\alpha^k<\infty,$$
completing the proof.
\end{proof}

The next lemma asserts that $Y_n$ is in fact a good approximation of $X_n$. 


\begin{lemma} \label{estimation}
There is a constant $C>0$ such that for all $n\in\mathbb{N}$, we have
\begin{equation} \label{eq:L2-estimate}
\EE[(X_n-Y_n)^2]\le C\alpha^n.
\end{equation}
\end{lemma}

\begin{proof}
 We employ a similar strategy as in Lemma \ref{L2 bounded}. First define the notation
 $$S(\omega):=\sum_{\gamma\in Q_n(\omega)}\tilde{t}_{\gamma},$$
 and notice that for each $\omega\in\mathcal{I}_n$ we have $\EE(S(\omega))=1$ by the proof of Corollary \ref{key id 2}. Furthermore, $S(\omega)$ is independent of $\mathscr{F}_n$. Next, we calculate
 \begin{align*}
     \EE[(X_n-Y_n)^2\big|\mathscr{F}_n] & = \EE\Bigg[\Bigg(\sum_{\omega\in Q_n}t_\omega -\sum_{\omega\in \mathcal{I}_{n}}t_\omega \Bigg)^2\:\Big|\:\mathscr{F}_n\Bigg]\\
     & = \EE\Bigg[\Bigg(\sum_{\omega\in \mathcal{I}_{n}}t_\omega \Big(S(\omega)-1\Big)\Bigg)^2\:\Big|\:\mathscr{F}_n\Bigg]\\
     & = \EE\Bigg[\sum_{\omega\in \mathcal{I}_{n}}t_\omega^2 \Big(S(\omega)-1\Big)^2 + \\
     & \qquad\quad + {\sum_{\substack{\sigma,\tau\in\mathcal{I}_n\\\sigma\ne\tau}}t_\sigma t_\tau \Big(S(\sigma)-1\Big)\Big(S(\tau)-1\Big)\Big|\mathscr{F}_n \Bigg]}
 \end{align*}
The double sum vanishes because for each $\sigma,\tau\in\mathcal{I}_n$ with $\sigma\ne\tau$ we have
 \begin{align*}
      \EE\Big[t_{\sigma}t_\tau (S(\sigma)-1)(S(\tau)-1)\big|\mathscr{F}_n \Big]
     & = t_{\sigma}t_\tau\EE\Big[(S(\sigma)-1)(S(\tau)-1)\Big]\\
     & = t_{\sigma}t_\tau\EE(S(\sigma)-1)\EE(S(\tau)-1)\\
     & = 0.
 \end{align*}
 Here, the first equality follows since $t_\sigma$ and $t_\tau$ are $\mathscr{F}_n$-measurable and $S(\sigma)$ and $S(\tau)$ are independent of $\mathscr{F}_n$, and the second equality comes from the fact that $\sigma$ and $\tau$ have split before level $n$, and thus $S(\sigma)$ and $S(\tau)$ are independent of each other. As a result, we have
\begin{equation} 
\EE[(X_n-Y_n)^2\big|\mathscr{F}_n]  = \EE\Bigg[\sum_{\omega\in \mathcal{I}_{n}}t_\omega^2 \Big(S(\omega)-1\Big)^2\:\Big|\:\mathscr{F}_n\Bigg]
      =\sum_{\omega\in \mathcal{I}_{n}}t_\omega^2 \EE\Big[(S(\omega)-1)^2\Big]. 
\label{eq:the big one}
\end{equation}
 
We next estimate $\EE\big[(S(\omega)-1)^2\big]$ for each $\omega\in\mathcal{I}_n$. Note that this expectation is just the variance of $S(\omega)$. We show that it is bounded by a uniform constant, as follows:
 \begin{align*}
     \EE\Big[(S(\omega)-1)^2\Big] 
		 & = \EE\Big[(S(\omega))^2\Big]-1\\
     & = \EE\left(\sum_{\gamma\in Q_n(\omega)}\tilde{t}_{\gamma}^{2}\right)+\EE\left(\sum_{\substack{\sigma,\tau\in Q_n(\omega)\\\sigma\ne\tau}}\tilde{t}_{\sigma}\tilde{t}_{\tau}\right)-1\\
     & \le \EE\left(\sum_{\substack{\sigma,\tau\in Q_n(\omega)\\\sigma\ne\tau}}\tilde{t}_{\sigma}\tilde{t}_{\tau}\right).
 \end{align*}
Here we used the fact that $\tilde{t}_{\gamma}^{2}\le \tilde{t}_{\gamma}$ in combination with $\EE(S(\omega))=1$. 
Now, in order to estimate the right side of the last inequality, we split the terms in the sum according to the last common ancestor of the words $\sigma$ and $\tau$.
To that end we define 
$$W(\omega):=\{\gamma\in\mathcal{I}^*:|\gamma|\ge n, \text{ and }\gamma\text{ is a prefix of some }\gamma'\in Q_n(\omega)\}.$$
Note that $\omega\in W(\omega)$. Next we define the set 
\[
U(\gamma):=\left\{(\sigma,\tau)\in Q_n(\omega)\times Q_n(\omega):\ \sigma\wedge\tau=\gamma\right\}, 
\]
where $\sigma\wedge\tau$ denotes the longest common prefix of $\sigma$ and $\tau$.
Now observe that
\begin{equation}
    \begin{split}
    \label{the ugly one}
    &\EE\left(\sum_{\substack{\sigma,\tau\in Q_n(\omega)\\\sigma\ne\tau}}\tilde{t}_{\sigma}\tilde{t}_{\tau}\right)
      = \sum_{\gamma\in W(\omega)}\sum_{(\sigma,\tau)\in U(\gamma)}\EE(\tilde{t}_{\sigma}\tilde{t}_{\tau})\\
    & \qquad = \sum_{\gamma\in W(\omega)}\sum_{(\sigma,\tau)\in U(\gamma)}\EE\big(\tilde{t}_\gamma^{2}\big)\EE\big(a_{\sigma|_{|\gamma|+1}}l_{\sigma_{|\gamma|+1}}^{s-1}a_{\tau|_{|\gamma|+1}}l_{\tau_{|\gamma|+1}}^{s-1}\big)\\
    & \qquad\qquad\quad \times\EE\left(\prod_{i=|\gamma|+2}^{|\sigma|}a_{\sigma|_i}l_{\sigma_i}^{s-1}\right)\EE\left(\prod_{i=|\gamma|+2}^{|\tau|}a_{\tau|_i}l_{\tau_i}^{s-1}\right)\\
    & \qquad \le C'\sum_{\gamma\in W(\omega)}\EE\big(\tilde{t}_\gamma^{2}\big)\sum_{(\sigma,\tau)\in U(\gamma)}\EE\left(\prod_{i=|\gamma|+1}^{|\sigma|}a_{\sigma|_i}l_{\sigma_i}^{s-1}\right)\EE\left(\prod_{i=|\gamma|+1}^{|\tau|}a_{\tau|_i}l_{\tau_i}^{s-1}\right),
    \end{split}
\end{equation}
where
\[
C':=\max\left\{\frac{\EE(a_k a_l)}{\EE(a_k)\EE(a_l)}:\ (k,l)\in\mathcal{I}_1\times\mathcal{I}_1\right\}.
\] 
Using that $U(\gamma)\subset Q_n(\gamma)\times Q_n(\gamma)$, we estimate the double sum over $(\sigma,\tau)$ in the last expression in \eqref{the ugly one} by
\begin{align*}
    & \sum_{(\sigma,\tau)\in U(\gamma)}\EE\left(\prod_{i=|\gamma|+1}^{|\sigma|}a_{\sigma|_i}l_{\sigma_i}\right)\EE\left(\prod_{i=|\gamma|+1}^{|\tau|}a_{\tau|_i}l_{\tau_i}\right)\\   
    & \qquad\qquad \le \sum_{\sigma,\tau\in Q_n(\gamma)}\EE\left(\prod_{i=|\gamma|+1}^{|\sigma|}a_{\sigma|_i}l_{\sigma_i}\right)\EE\left(\prod_{i=|\gamma|+1}^{|\tau|}a_{\tau|_i}l_{\tau_i}\right)\\
    & \qquad\qquad = \big(E(S(\gamma))\big)^2=1.
\end{align*}
Next, we group the elements of $W(\omega)$ by their length, to obtain the estimate
\begin{equation}
    \begin{split}
        \label{the last one}
        \sum_{\gamma\in W(\omega)}\EE\big(\tilde{t}_\gamma^{2}\big) & \le\sum_{j=n}^\infty\sum_{\gamma\in\mathcal{I}_j,\,\gamma|_n=\omega}\EE\big(\tilde{t}_\gamma^{2}\big)
         = \sum_{j=n}^\infty\sum_{\gamma\in\mathcal{I}_{j-n}}\EE\big(t_\gamma^{2}\big)\\
        & = \sum_{j=n}^\infty\alpha^{j-n}=\frac{1}{1-\alpha},
    \end{split}
\end{equation}
where the first equality follows since $h_{\gamma_{n+1},\dots,\gamma_{j}}$ has the same distribution as $\tilde{h}_{\gamma}$.

Substituting the last two estimates into \eqref{the ugly one} and combining with (\ref{eq:the big one}) we see that
\begin{equation*}
    \EE[(X_n-Y_n)^2\big|\mathscr{F}_n] = \sum_{\omega\in \mathcal{I}_{n}}t_\omega^2 \EE\Big[(S(\omega)-1)^2\Big]
     \le \frac{C'}{1-\alpha}\sum_{\omega\in \mathcal{I}_{n}}t_\omega^2.
\end{equation*}
Put $C:=C'/(1-\alpha)$, then by taking expectation of both sides we obtain \eqref{eq:L2-estimate}.
\end{proof}


\begin{lemma} \label{counting of boxes}
We have the inequalities
$$\frac{\delta^{-ns}}{\delta}X_n\le N_{\delta^n}(G)\le \delta^{-ns}X_n+2\delta^{-(n+1)}.$$
\end{lemma}

\begin{proof}
 As in the proof of \cite[Proposition 11.1]{falconer}, we estimate the number of boxes $N_{\delta^n}(G)$ by adding the heights of the rectangles in $Q_n$ and dividing by the box size $\delta^n$.
 To be more precise, we first notice that this way of counting could have missed out on the top and bottom of each rectangle, and therefore we have an upper bound of 
 $$N_{\delta^n}(G)\le2\#Q_n+\sum_{\omega\in Q_n}h_\omega\delta^{-n},$$
where $\#Q$ denotes cardinality. Because each $\omega\in Q_n$ satisfies $\delta^{n+1}\le l_\omega\le\delta^{n}$, we have that $\#Q_n\le\delta^{-(n+1)}$. 

On the other hand, 
any single box of width $\delta^n$ can meet at most $\delta^{-1}$ many rectangles, since the width of each rectangle is at least $\delta^{n+1}$. Thus
$$\frac{1}{\delta}\sum_{\omega\in Q_n}h_\omega\delta^{-n}\le N_{\delta^n}(G)\le 2\delta^{-(n+1)}+\sum_{\omega\in Q_n}h_\omega\delta^{-n}.$$
Multiplying the left and right sides of this inequality by $\delta^{ns}\delta^{-ns}$ and using the assumptions on the lengths $l_\omega$ for $\omega\in Q_n$, gives the result.
\end{proof}


\begin{proof}[Proof of Theorem \ref{box thm}]
By Lemmas \ref{martingale} and \ref{L2 bounded} and the martingale convergence theorem (see \cite{williams}), there exists a non-negative random variable $Y$ such that (almost surely) $Y_n\to Y$, and $\EE(Y)=\EE(Y_0)=1$. 
In fact, $Y$ is positive almost surely. Indeed, since $\EE(Y)>0$, $Y=0$ with some probability $p<1$. However, $Y=0$ if and only if the sums through each branch converge to zero. That is 
$$\sum_{\omega\in\mathcal{I}_n^k}h_\omega l_\omega^{s-1}\to 0\quad \mbox{as} \quad n\to\infty$$
for each $k=0,1,\dots,m-1$, where $\mathcal{I}_n^k$ denotes the set of words $\omega\in\mathcal{I}_n$ such that $\omega_1=k$. However, each of these events also happens with probability $p$ independently of one another.
Therefore $p^{m}=p$, which implies $p=0$ since $p<1$. In particular there exist, with probability one, positive (random) constants $M<N$ such that $M<Y_n<N$ for all $n$.

Now Lemma \ref{estimation} implies
$$\sum_{n=1}^\infty \PP(|X_n-Y_n|\ge\eps)\le\frac{C}{\eps^2}\sum_{n=0}^{\infty}\alpha^n<\infty$$
for any $\eps>0$, by Chebyshev's inequality. By the Borel-Cantelli lemma, the probability that $|X_n-Y_n|\ge\eps$ infinitely often is zero. But $\eps$ was arbitrary, so almost surely $|X_n-Y_n|\to0$. Hence, with probability one, 
\begin{equation} \label{eq:Z-sandwich}
M<X_n<N \qquad\mbox{for all sufficiently large $n$}.
\end{equation}
Thus, we have
\begin{align*}
    \dim_B G & = \lim_{n\to\infty}\frac{\log N_{\delta^n}(G)}{-n\log\delta}
     = \lim_{n\to\infty}\frac{\log(\delta^{-ns}X_n)}{-n\log\delta}\\
    & = s-\lim_{n\to\infty}\frac{\log X_n}{n\log\delta} = s.
\end{align*}
The first equality comes from the discussion following Definition 3.1 in \cite{falconer}; namely, it is sufficient to consider the limit along an exponentially decreasing sequence. The second equality follows by Lemma \ref{counting of boxes}, and the last equality follows from \eqref{eq:Z-sandwich}.
\end{proof}

\section*{Funding}
Allaart was partially supported by the Simons Foundation [\#709869].


\end{document}